\documentclass[a4paper,12pt]{amsart}

\usepackage{amsmath,amsfonts,amsthm,amsopn,cite,mathrsfs}
\usepackage{epsfig,verbatim}
\usepackage{subfigure,color}
\newtheorem{tm}{Theorem}[section]    % theorem like environments
\newtheorem{lemma}[tm]{Lemma}
\newtheorem{prop}[tm]{Proposition}
\newtheorem{cor}[tm]{Corollary}

\theoremstyle{definition}
\newtheorem{df}{Definition}[section]

\newtheorem{rem}{Remark}[section]
\newcommand{\beqa}{\begin{eqnarray*}}
\newcommand{\eeqa}{\end{eqnarray*}}

\DeclareMathOperator*{\supp}{supp}

\def\cO{\mathcal{ O}}

\def\<{\left<}
\def\>{\right>}

\def\mv1{M_v^1}

\usepackage[a4paper,left=3cm,right=3cm,top=3.5cm,bottom=4.5cm]{geometry}

\usepackage{url}
\usepackage{amssymb}

\newcommand{\EMN}{Erd\'elyi-Magnus-Nevai}

\newcommand{\N}{\ensuremath{\mathbb{N}}}
\newcommand{\mz}{Marcinkiewicz-Zygmund}
\newcommand{\sd}{\mathbb{S}^d}

\renewcommand{\S}{\ensuremath{\mathbb{S}}}

\newcommand{\mut}{\mu _t}
\newcommand{\pit}{\Pi _t}

\newcommand{\dimp}{\mathrm{dim}\, \pit}

\newcommand{\ty}{^t_y}

\DeclareMathOperator{\vol}{vol}
\DeclareMathOperator{\diam}{diam}

\begin{document}
\begin{abstract}
We derive quantitative volume constraints for sampling measures $\mu_t$ on the unit sphere $\S^d$ that satisfy \mz\ inequalities of order $t$. Using precise localization estimates for Jacobi polynomials, 
we obtain explicit upper and lower bounds on the $\mu_t$-mass of
geodesic balls at the natural scale $t^{-1}$. Whereas  constants are
typically left implicit in the literature,  we place special emphasis
on fully explicit constants,  and the results are   genuinely quantitative.  Moreover, these bounds yield quantitative constraints for the $s$-dimensional Hausdorff volume of \mz\ sampling sets and, in particular, optimal lower bounds for the length of \mz\ curves.
\end{abstract}

\title{Quantitative Constraints for Stable Sampling on the Sphere}
\author{Martin Ehler}
\address{Faculty of Mathematics \\
University of Vienna \\
Oskar-Morgenstern-Platz 1 \\
A-1090 Vienna, Austria}
\email{martin.ehler@univie.ac.at}

\author{Karlheinz Gr\"ochenig}
\address{Faculty of Mathematics \\
University of Vienna \\
Oskar-Morgenstern-Platz 1 \\
A-1090 Vienna, Austria}
\email{karlheinz.groechenig@univie.ac.at}
\subjclass[2010]{}
%\date{31.10.2025}

\maketitle

\section{Introduction}
\mz\ inequalities are often referred to as ``sampling inequalities'' and form the foundation of
stable sampling, see \cite{Groechenig2020,OrtegaSaludes} and the references therein. They appear, sometimes under different names, in several areas of mathematics. In approximation theory, for instance, they are also known as approximate quadrature rules or approximate $t$-designs, and are special types of frame inequalities. In complex analysis these conditions correspond to Logvinenko-Sereda sets and Carleson measures \cite{OrtegaPridhnani2013}.

Necessary conditions for \mz\ inequalities have been known for some time on 
the torus and the sphere \cite{Marzo2007,MNW01,Reimer}, and more generally on 
Riemannian manifolds \cite{OrtegaPridhnani2012,OrtegaPridhnani2013} and on metric measure spaces 
\cite{FilbirMhaskar2011,Mhaskar2018}. However, these results are
typically qualitative, and  
the relevant constants are either left 
unspecified \cite{Brauchart2015,FilbirMhaskar2011,OrtegaPridhnani2013} or do not appear in density formulations \cite{FuehrAND,Groechenig2017,Marzo2007,OrtegaPridhnani2012}. For numerical work, explicit estimates are essential, especially on the $d$-sphere, where the dependence on the dimension $d$ plays a critical role.

Several quantitative \emph{sufficient} conditions for \mz\ inequalities on the sphere are known \cite{FHJU, Filbir:Themi,Lindblad3,MarzoPridhnani}. 
In contrast, the situation for necessary conditions is different, and the quantitative picture has remained incomplete. 

In the present work we complete this picture by providing fully quantitative
estimates on the sphere, with constants depending explicitly on the
dimension. Although the results build on classical ideas and are informed by the qualitative density framework developed in \cite{Marzo2007,marzo08}, the
derivation of  explicit bounds is a genuine difficulty, since many standard arguments suppress constants at crucial steps. Moreover, the quantitative constraints for curves and general $s$-dimensional sampling sets constitute a key new contribution that goes well beyond the classical point-sampling framework.

To state our results precisely, we now introduce the necessary technical framework.

\subsubsection*{Sampling inequalities}
Let $\Pi_t$ be the $\binom{t+d}{t}+\binom{t-1+d}{t-1}$-dimensional space 
of spherical harmonics of degree at most $t$. Given a finite Borel measure $\mu_t$ on $\S^d$, the lower and upper 
\mz\ inequalities of order $t$ (or for $\Pi_t$) correspond to the two-sided bound
\begin{equation}\label{eq:def mz}
a \|f\|_{L^2(\S^d)}^2 
   \stackrel{\text{lower}}{\le} \int_{\S^d} |f(x)|^2 \, d\mu_t(x)
   \stackrel{\text{upper}}{\le} A \|f\|_{L^2(\S^d)}^2,
   \qquad f \in \Pi_t .
\end{equation}
Such inequalities guarantee that $\mu_t$ samples $\Pi_t$ in a stable manner 
\cite{MNW01}, with stability constants controlled by $a$ and~$A$. In
relevant applied  
scenarios, the function values of $f$ are collected either  (i)  on  a finite
set,  so that $\mu_t$ is a 
discrete measure, or (ii) $\supp \mu_t$ is an  $s$-dimensional set with $0<s<d$. 

\subsubsection*{Curves}
Disregarding pathologies, one-dimensional subsets of $\sd $ are
curves, and sampling  along curves is referred to as  mobile sampling
on the sphere. It  has emerged as an alternative to 
point-based quadrature rules, where information is collected along a moving 
sensor trajectory rather than at discrete locations 
\cite{Ehler2021,Groechenig2015,Hastie,Lee,Ilya23,Vetterli1,Vetterli2}. 
Such curves arise naturally in scanning and sensing applications, including 
satellite paths, drone sweeps, and tomographic scanning. In this setting, the 
integral $\int_{\sd} |f(x)|^2\, d\mu_t(x)$ in \eqref{eq:def mz} becomes a 
normalized line integral along the curve $\gamma_t$ with a weight function $w_t$,
\begin{equation*}
\int_{\S^d} |f(x)|^2 \, d\mu_t(x)=\frac{1}{\ell(\gamma_t)}\int_{\gamma_t} |f|^2 w_t\,,
\end{equation*}
where $\ell(\gamma_t)$ denotes the arc-length of $\gamma _t$. 
In the special case $a=A=1$ and $w_t \equiv 1$, we speak of
spherical $t$-design curves \cite{EG:2023,Lindblad1,Yang2026} and the \mz\
inequality \eqref{eq:def mz} is  equivalent to a quadrature rule on the sphere. 

As shown in our earlier work \cite{EG:2023}, the length of a spherical $t$-design curve on 
$\S^d$ must satisfy\footnote{We write
  $f \lesssim g$ if 
  $f \leq Cg$. If $\lesssim$ and $\gtrsim$ both hold, then we
  write $\asymp$. In some case it is more convenient to write $f = \cO
  (g)$ for $f\lesssim g$, and we will use this convention, too. Other communities would use $\Omega$ and $\Theta$, respectively.} 
\begin{equation}\label{eq:l t design}
\ell(\gamma_t) \gtrsim t^{d-1}\,.
\end{equation}
We note that $\dim\Pi_t\asymp t^d$, so that the bound also says $\ell(\gamma_t) \gtrsim \dimp/t$.  
This estimate  raises the question of whether we can determine
explicit constants and whether comparable constraints  also 
hold for general \mz\ curves and for higher-dimensional sampling sets. 

\subsubsection*{Questions}
We are guided by two questions, and although \eqref{eq:l t design}
provides initial motivation, we begin with more general questions:
\begin{enumerate}
    \item[(Q.1):\,] 
    How concentrated can a \mz\ sampling measure $\mu_t$ of order $t$
    be at the resolution scale $t^{-1}$?
    In other words, what is the maximal possible $\mu_t$-mass of a
    spherical cap  of radius $\mathcal{O}(t^{-1})$?

    \item[(Q.2):\,]
    What volume obstructions do \mz\ inequalities impose on sampling sets
    of fixed Hausdorff dimension?
    In particular, what are the minimal $s$-dimensional volumes and, in
    the case $s=1$, the minimal lengths of \mz\ sampling sets of order $t$?
\end{enumerate}

\subsubsection*{Answers}
While qualitative answers to these questions are known, we will derive
precise,  quantitative estimates to   both  
questions with explicit constants. In particular, we show that the qualitative length 
constraint \eqref{eq:l t design}  for spherical $t$-design curves
extends to general  \mz\ families of measures. For the first time,  we obtain quantitative bounds for $s$-dimensional sampling sets. More specifically, we provide the following answers:
\begin{itemize}
\item[(A.1)] 
Every measure $\mu_t$ 
satisfying the \emph{upper} \mz\ inequality cannot assign too much mass to a 
single spherical cap $B_r$ of radius $r$ at scale $t^{-1}$,
\begin{equation*}
\mu_t(B_r) \leq C_1 \frac{A}{\dimp}\,\qquad \text{for } r\leq
C_2t^{-1}\, .
\end{equation*}
The  constants $C_1=C_1(d),\,C_2=C_2(d)$  depend only on the dimension
$d$ and will be made explicit in 
Theorem \ref{mutmeas}. In particular, on the sphere  this provides a quantitative
counterpart  to the generic estimates in \cite[Theorem 5.5 (a)]{FilbirMhaskar2011}. 

We recommend the reader to consult the work of Marzo and
Ortega-Cerd\`a~\cite{marzo08} where measures satisfying the upper \mz\
inequality  are characterized
by their relative density. While the proofs and results are definitely related, see also \cite{Marzo2007},
our focus is on the explicit constants. Note that these authors also
call such measures  Logvinenko-Sereda measures or speak of a Carleson
family.

\item[(A.2a)] 
If $(\gamma_t,w_t)$ is a weighted \mz\ curve and $t\geq d\geq 2$, then its length is lower bounded by
\[
\ell(\gamma_t)\geq  \frac{\dimp}{t} \,\frac{\sqrt{d+2}}{4}\,
\Big(\frac{a}{A}\Big)^{1+\frac{1}{2t}}\,.
\]
Since  $\frac{\dimp}{t}\asymp t^{d-1}$, this estimate also answers a
question of Lindblad~\cite[Sec.~6]{Lindblad2} whether the lower bounds
on the length of $t$-design curves extend to approximate design
curves. Our result not only  matches \eqref{eq:l t design}, but  
makes the involved constants and their dependence of the \mz\
constants and the dimension explicit.

\item[(A.2b)]
Let $X_t\subseteq \S^d$ be an $s$-dimensional subset and $w_t$
its associated nonnegative 
weight function. Assume that  the upper inequality in \eqref{eq:def mz} 
holds with the integral replaced by 
\[
\int_{\S^d} |f(x)|^2 \, d\mu_t(x)=\frac{1}{h_{s}(X_t)}\int_{X_t} |f(x)|^2 w_t \,d h_{s}(x)\,,
\]
where $h_s$ is the $s$-dimensional Hausdorff measure. If $(X_t,w_t)$ satisfies 
a mild Ahlfors regularity condition in a single point, then the Hausdorff volume must satisfy
\[
h_{s}(X_t) \geq c \, \frac{\dimp}{t^{s}}\,.
\]
The constant $c$ is specified in Theorem \ref{tm:s-dim}. For point sets, the regularity condition is automatically satisfied, and for curves it is guaranteed by the lower \mz\ inequality, cf.~Remarks \ref{rem:point zero} and \ref{rem:curve and Alhfors}.
\end{itemize}

Whereas (A.1) gives a quantitative answer to (Q.1), (A.2) resolves
(Q.2)  for curves 
and, even more generally, for sets of arbitrary  Hausdorff dimension
$s$. Since $\frac{\dimp}{t^{s}}\asymp t^{d-s}$, it extends \eqref{eq:l t design} 
to the full range of dimensions $s$. Moreover, (A.2a)  is optimal up to a dimension-dependent constant: there exist \mz\ curves $\gamma_t$, whose lengths satisfy $\ell(\gamma_t)\asymp t^{d-1}$, see \cite{EGK:24,Lindblad2}.

\subsubsection*{Proofs via quantitative localization}
We remark that the
original proof of \eqref{eq:l t design} in~\cite{EG:2023} relies on the exactness of the
quadrature rule and does not extend to \mz\ measures. In this work we
develop a new approach whose  key ingredient  is a quantitative
localization principle  
of reproducing kernels for spherical harmonics at the natural scale
$t^{-1}$, which is  obtained from explicit estimates of 
Jacobi polynomials. Using these well-localized kernels as test 
functions in $\Pi_t$, we derive the upper volume bound (A.1). The
proofs of (A.2a) and (A.2b) then rely on (A.1).

In the last part, we  establish lower volume bounds complementary to
(A.1), and  
show  that every \mz\ measure $\mu_t$ must assign sufficient mass to each 
spherical cap,
\begin{equation}\label{eq:2}
\mu_t(B_R) \geq C_3\frac{a}{\dimp}\,,\qquad \text{for } R\geq C_4 t^{-1}\,,
\end{equation}
with constants $C_3=C_3(d),\,C_4=C_4(d)$ made explicit in Theorem~\ref{main}. 
In particular, this recovers a portion of \cite[Theorem 5.5 (b)]{FilbirMhaskar2011}, but now with explicit constants for the sphere.

The bounds \eqref{eq:2} and (A.1) control the local mass distribution of \mz\ measures: 
(A.1) prevents concentration, and \eqref{eq:2} excludes large  holes
in the sampling manifold. Taken in combination, 
they imply that \mz\ measures quantitatively resemble the normalized surface 
measure at scale $t^{-1}$. However, it should be noted that the
constant $C_4$ grows rapidly with the dimension $d$, cf.~\eqref{eq:Vd}
in Theorem \ref{main}. While this may appear unsatisfactory at first
sight, it may equally well be regarded as a noteworthy difficulty and
as a new problem  that
emerges only by the explicit tracking of constants.

The analysis of \eqref{eq:2} relies on the deep bounds of \EMN\ \cite{EMN} for Jacobi polynomials. These estimates are essential for controlling the off-diagonal decay of the associated kernels with the required precision, which is yet another instance of kernel localization.

\subsubsection*{The paper is organized as follows.}
In Section \ref{sec:pre} we introduce the setting and recall the reproducing 
kernel for the polynomial space $\Pi_t$ in terms of Jacobi polynomials. 
Quantitative localization estimates for Jacobi polynomials are derived in 
Section \ref{sec:localization I}. In Section \ref{sec:consequences} we
derive 
the upper volume bound (A1) as a consequence of the upper \mz\ 
inequality and record some elementary implications of the lower \mz\ 
inequality. Lower bounds on the length of weighted \mz\ curves are established 
in Section \ref{sec:curves}, while $s$-dimensional sampling sets are treated in 
Section \ref{sec:Hausdorff}. Finally, Section \ref{sec:prep} provides the 
necessary off-diagonal kernel estimates, which are then used in 
Section \ref{sec:lower volume bounds} to derive the lower volume bound 
\eqref{eq:2}.

\section{Preliminaries}\label{sec:pre}
In this section we fix notation and 
summarize the basic facts about  polynomial subspaces  on the
sphere. 

\subsection{Notation and spherical geometry}
The standard inner product in $\mathbb{R}^{d+1}$ is denoted by $x\cdot y$. It induces the norm $\|x\|^2=x\cdot x$, so that the unit sphere is $\S^d=\{x\in\mathbb{R}^{d+1}: \|x\|=1\}$. The standard metric on the sphere measures distance between $x$ and $y$ in $\S^d$ by the angle $\phi\in[0,\pi]$ satisfying $\cos \phi=x\cdot y$. A closed ball on the sphere of radius $r\in[0,\pi]$ centered at $x$ is then the spherical cap
\begin{equation*}
B_r (x)= \{y\in \sd : x\cdot y \geq \cos r \}\,.
\end{equation*}
Thus $y\in B_r (x)$ if and only if the angle $\phi$ between $x$ and $y$ satisfies $0\leq \phi\leq r$. 

We endow $\S^d$ with the normalized surface measure, so that the volume of a subset $S\subseteq \S^d$ may be denoted by $|S|=\int_{S}1 dx$ with $|\S^d|=1$. The inner product between $f,g\in  L^2(\S^d)$ is written as $\langle f,g\rangle_{L^2(\S^d)}=\int_{\S^d} f(x)\overline{g(x)}dx$ and induces the norm $\|f\|_{L^2(\S^d)}$. 

All functions $f$ we consider are spherical polynomials, and we next summarize basic algebraic properties of the space of polynomials on the sphere of degree $t$. 

\subsection{Polynomials on the sphere and reproducing kernels}

The subspace $\pit \subseteq L^2(\S^d)$ consists of all algebraic polynomials in $d+1$ variables of degree at most $t\in\N$ restricted to the sphere. Its dimension is 
\begin{equation}\label{eq:dimp}
\dimp = \binom{d+t}{t} + \binom{d+t-1}{t-1} = \binom{d+t}{t} \,\frac{2t+d}{t+d}\,.
\end{equation}

Let $k^t:\S^d\times\S^d\rightarrow \mathbb{R}$ denote the reproducing kernel of $\Pi_t$. According to the convention $k^t_x(y) :=k^t(x,y)$, the kernel reproduces every $f\in\Pi_t$ by 
\begin{equation*}
f(x) = \langle f,k^t_x \rangle_{L^2(\S^d)}\,.
\end{equation*}

Since the space $\Pi_t$ is orthogonally invariant, so is its reproducing kernel $k^t$. It can be expressed by Jacobi polynomials $P^{(\alpha,\beta)}_t$ that are orthogonal 
on $[-1,1]$ with respect to the weight function $(1-u)^{\alpha}(1+u)^{\beta}$ and 
normalized by\footnote{We use the Gamma-function extension of the binomial coefficient and set $\binom{x}{y} = \frac{\Gamma(x+1)}{\Gamma(y+1)\Gamma(x-y+1)}$ 
for real \(x,y\) whenever the right-hand side is finite.} 
\begin{equation*}
P^{(\alpha,\beta)}_t(1) = \binom{t+\alpha}{t}\,.
\end{equation*}
Using the (re-)normalized Jacobi polynomials 
\begin{equation*}
\widetilde{P}^{(\alpha,\beta)}_t:=P^{(\alpha,\beta)}_t / P^{(\alpha,\beta)}_t(1)
\end{equation*} 
for the parameters $(\alpha,\beta)=(\frac{d}{2},\frac{d}{2}-1)$, the reproducing kernel has the clean representation 
 \begin{equation}\label{eq:def k in P}
k^t(x,y) =  \dimp \, \widetilde{P}_t^{(\frac{d}{2},\frac{d}{2}-1)}(x\cdot y)\,,\qquad x,y\in\S^d\,,
\end{equation}
see \cite[(4.31)]{Reimer}. The $L^2$-norm of the polynomial $k^t_x\in\Pi_t$ is 
\begin{equation}\label{eq:kt in L2}
\|k^t_x\|^2_{L^2(\S^d)}= \langle k^t_x, k^t_x \rangle _{L^2(\S^d)} = k^t(x,x)=\dimp \,,\qquad \forall\, x\in\S^d\,.  
\end{equation}
Instead of $k^t$, it is sometimes more convenient to directly use the normalized Jacobi polynomial 
\begin{equation*}
f^t(x,y):=\widetilde{P}_t^{(\frac{d}{2},\frac{d}{2}-1)}(x\cdot y) = \frac{k^t(x,y)}{\dimp}
\end{equation*}
that satisfies $f^t(x,x)=1$ and 
\begin{equation}\label{eq:norm of f}
\|f^t_x\|_{L^2(\S^d)}^2 = \frac{1}{\dimp}\,.
\end{equation}

Having established these algebraic features, we now shift towards analytic properties of Jacobi polynomials.

\section{Localization of the Jacobi polynomials}\label{sec:localization I}
In this section we quantify the localization properties of the Jacobi polynomials that generate the reproducing kernel $k^t$ of $\Pi_t$, cf.~\eqref{eq:def k in P}. These estimates capture how strongly the kernel concentrates near the diagonal and provide the key analytic input for the volume bounds established in the following sections. Although not strictly needed for our purposes, we present the arguments for the full range of parameters $\alpha,\beta$.

\begin{lemma} \label{l0}
If $\alpha\geq \beta> -1$ and $t\in\N$, then  we have
   \begin{equation}\label{eq:pPp}
   \Big|  \widetilde{P}^{(\alpha,\beta)}_t(\cos \phi) -1\Big|
      \leq  \frac{t\,(t+\alpha+\beta+1)}{2(\alpha+1)} \, \phi^2 \,.
   \end{equation}
 \end{lemma}

 \begin{proof}
 For $t=0$, the inequality is trivially satisfied due to $\widetilde{P}^{(\alpha,\beta)}_0\equiv 1$. We now suppose that $t\geq 1$ and define the function $p(\phi)=\widetilde{P}^{(\alpha,\beta)}_t(\cos \phi)$. The mean value theorem of differential calculus yields
 \begin{align*}
 |p(\phi)-p(0)| & \leq |\phi| \sup_{|\theta|\leq|\phi|}|p'(\theta)|\\
 & = |\phi| \sup_{|\theta|\leq|\phi|} \Big| \big(\widetilde{P}^{(\alpha,\beta)}_t\big)'(\cos\theta)\,\,\sin\theta\Big|\,.
 \end{align*}
 Since $p(0)=\widetilde{P}^{(\alpha,\beta)}_t(1)=1$ and $|\sin\theta|\leq |\theta|$, we have
 \begin{equation}\label{eq:midterms}
 \Big| \widetilde{P}^{(\alpha,\beta)}_t(\cos \phi)-1\Big| \leq |\phi|^2 \sup_{|\theta|\leq|\phi|} \Big| \big(\widetilde{P}^{(\alpha,\beta)}_t\big)'(\cos\theta)\Big|\,.
 \end{equation}
The derivative of $P^{(\alpha,\beta)}_t$ is the scaled Jacobi polynomial 
\begin{equation*}
\big(P^{(\alpha,\beta)}_t\big)'=\frac{t+\alpha+\beta+1}{2} P_{t-1}^{(\alpha+1,\beta+1)}\,,
\end{equation*}
see \cite[18.9.15]{NIST:DLMF}. If $\alpha+1\geq \beta+1> -1$ and
$\max(\alpha+1,\beta+1)\geq -\frac{1}{2}$, then, according to \cite[Theorem 7.32.1]{Szego1975}, the maximum of $\big|P_{t-1}^{(\alpha+1,\beta+1)} \big|$ on $[-1,1]$ is attained in $1$, i.e., 
\begin{equation*}
\max_{u\in[-1,1]}\big| P_{t-1}^{(\alpha+1,\beta+1)}(u)\big| = P_{t-1}^{(\alpha+1,\beta+1)}(1) = \binom{t+\alpha}{t-1}\,.
\end{equation*}
Therefore, the derivative of the normalized Jacobi polynomial satisfies
\begin{align*}
\Big|\big(\widetilde{P}^{(\alpha,\beta)}_t\big)' (\cos\theta) \Big| &=\Big|\frac{1}{\binom{t+\alpha}{t}}\,\frac{t+\alpha+\beta+1}{2} P_{t-1}^{(\alpha+1,\beta+1)}(\cos\theta)\Big|\\
&\leq  \Big|\frac{(t+\alpha+\beta+1)\binom{t+\alpha}{t-1}}{2\binom{t+\alpha}{t}}\Big|\\
& = 
\Big|\frac{t\,(t+\alpha+\beta+1)}{2(\alpha+1)}\Big|\,.
\end{align*}
We substitute this estimate  into \eqref{eq:midterms} and obtain the claim \eqref{eq:pPp}. 
\end{proof}

In the case $(\alpha,\beta)=(\frac{d}{2},\frac{d}{2}-1)$ 
the Jacobi polynomial induces the reproducing kernel
$k^t(x,y)=\dimp\,\widetilde{P}^{(\frac{d}{2},\frac{d}{2}-1)}_t(x\cdot
y)$, see \eqref{eq:def k in P}.  We now  introduce a 
radius $r$, so that the normalized Jacobi polynomial stays above a
target level $\nu$ for all angles $\phi\leq r$.
\begin{prop}\label{cor:conc}
Let $\nu \in (0,1)$. If the radius $r=r(d,t,\nu)>0$ is chosen by
\begin{equation}\label{eq:r}
r^2 =  (1-\nu)\frac{(d+2)}{t(t+d)}\,,
\end{equation}
then the normalized Jacobi polynomial obeys the  lower bound 
\begin{equation*}
 \widetilde{P}^{(\frac{d}{2},\frac{d}{2}-1)}_t(\cos \phi)  \geq  \nu\, \qquad \forall \, 0\leq \phi\leq r\,.
\end{equation*}
\end{prop}
This proposition provides a uniform lower bound on the 
reproducing kernel $k^t$ near the diagonal and shows that  the natural
scale is of the order $t^{-1}$.

\begin{proof}
According to Lemma \ref{l0}, we obtain $\widetilde{P}^{(\alpha,\beta)}_t(\cos \phi)\geq \nu$ for $0\leq\phi\leq r$ by choosing $r>0$ as the solution of the equation 
\begin{equation*}
1 - \frac{t\,(t+\alpha+\beta+1)}{2(\alpha+1)}\,r^2 = \nu\,.
\end{equation*}
This leads to 
\begin{equation*}
r^2 =  (1-\nu)\frac{2(\alpha+1)}{t(t+\alpha+\beta+1)}\,,
\end{equation*}
and $(\alpha,\beta)=(\frac{d}{2},\frac{d}{2}-1)$ yields $r$ as in \eqref{eq:r}.
\end{proof}

\section{Consequences of \mz\ inequalities}\label{sec:consequences}
In this section we extract quantitative consequences of the \mz\ inequalities. The resulting mass and support constraints form the basis for the length and volume estimates developed in the subsequent sections.

Recall that a Borel measure $\mu_t$ on the sphere $\S^d$ is called a \mz\ measure if there exist constants $a,A>0$ such that
\begin{equation*}
  a \|f\|_{L^2(\sd)}^2 \stackrel{\text{lower}}{\le} \int_{\sd} |f(x)|^2\, d\mu_t(x) \stackrel{\text{upper}}{\le} A \|f\|_{L^2(\sd)}^2
  \qquad \text{for all } f\in\Pi_t\,.
\end{equation*}
We refer to these estimates as the lower and upper \mz\ inequalities. Measures satisfying the upper inequality are also called Carleson measures in complex analysis.

  \subsection{Consequences of the upper \mz\ inequality}
For a measure $\mu_t$ satisfying the upper \mz\ inequality, we will derive an upper bound on $\mu_t\big(B_{r}(x)\big)$. 

Such an upper bound can be derived by a simple idea: for
every closed, nonempty subset $S\subseteq\S^d$ and every 
$f\in\Pi_t$, the following chain of inequalities holds,
\begin{align*}
\min_{y\in S}\big(|f(y)|^2\big) \,\mu_t(S) & \leq \int_{S} |f(y)|^2\, d\mu_t(y) \\
&\leq \int_{\S^d} |f(y)|^2\, d\mu_t(y)\\
&\leq A \|f\|_{L^2(\S^d)}^2 \, ,
\end{align*}
 where the last estimate is due to the upper \mz\ inequality. Therefore, the 
volume $\mu_t(S) $ must obey
\begin{equation}\label{eq:general upper bound for f}
\mu_t(S)\leq \frac{A\|f\|_{L^2(\S^d)}^2}{\min_{y\in S}|f(y)|^2} \qquad \forall f\in\Pi_t\,.
\end{equation}
The right-hand side does not depend on $\mu_t$ and may constitute a
reasonable bound if we feed  in a polynomial  $f \in \Pi_t$ that
is well-localized in $S$. The natural choice of $S$ is a spherical cap $B_r(x)$. 
\begin{tm}[upper volume bounds] \label{mutmeas}
Let $\nu\in(0,1)$ and assume that $\mu_t$ satisfies the upper \mz\ inequality of order $t$. If the radius $r>0$ is chosen by
 \begin{equation}\label{eq:r2}
r^2 =  (1-\nu)\frac{(d+2)}{t(t+d)}\,,
\end{equation}
then the upper volume bound
  \begin{equation}
    \label{eq:7}
    \mut \big(B_{r}(x)\big) \leq  \frac{A}{\nu ^2\,
      \dimp } 
  \end{equation}
  holds for all $x\in\S^d$. 
\end{tm}
Theorem \eqref{mutmeas} can be read on different levels.  Extracting
$\nu $ from \eqref{eq:r2} yields
$$
\nu = 1- r^2 \frac{t(t+d)}{d+2} \, ,
$$
so that the upper bound is just
$$
\mu _t \big(B_{r}(x)\big) \leq  \frac{A}{
  \dimp } \, \frac{1}{(1- r^2 \frac{t(t+d)}{d+2})^2}\, .
$$
This inequality holds for all $r$ for which 
$\nu \in (0,1)$, but it is meaningful only for small $r$ or $\nu $
close to $1$. 
By choosing $r= \rho \sqrt{d} / (2t)$, the right-hand side is bounded
by $\frac{A}{  \dimp } (1-2\rho ^2)^{-2}$. Thus, 
at the scale $r\asymp t^{-1}$, the resulting upper volume bound is of
order $A/t^d$. This substantially sharpens the necessary quadrature
bounds in \cite[Section 6.5]{Reimer}, which treat only the special
case $a=A=1$ for finitely supported measures and without explicit
constants. Our result applies to a far more general measure-theoretic
setting, only the one-sided upper \mz\ inequality is needed, and it
provides fully quantitative bounds. 

The proof is based on choosing the kernel $k^t_x=\dimp
\widetilde{P}^{(\frac{d}{2},\frac{d}{2}-1)}_t(x\cdot y)$ or
equivalently the Jacobi polynomial
$f^t_x(y)=\widetilde{P}^{(\frac{d}{2},\frac{d}{2}-1)}_t(x\cdot y)$ for
the right-hand side of  \eqref{eq:general upper bound for f}. 
\begin{proof}
Let $f^t(x,y)=\widetilde{P}^{(\frac{d}{2},\frac{d}{2}-1)}_t(x\cdot y)$. By \eqref{eq:norm of f}, the $L^2$-norm of $f^t_x$ is 
\begin{equation*}
\|f^t_x\|_{L^2(\S^d)}^2 = \frac{1}{\dimp}\,.
\end{equation*}
For $\cos\phi = x\cdot y$, the condition $y\in B_r(x)$ translates to $0\leq\phi\leq r$. 
By Proposition \ref{cor:conc} we have the lower bound 
\begin{equation}\label{eq:k^ unten}
| f^t_x(y)| = | \widetilde{P}^{(\frac{d}{2},\frac{d}{2}-1)}_t(\cos \phi)| \geq \nu \,,\qquad \forall \,y\in B_r(x)\,.
\end{equation}
Thus, \eqref{eq:general upper bound for f} leads to $\mut \big(B_{r} (x)\big) \leq \frac{A}{\nu ^2 \dimp }$.
\end{proof}

\subsection*{Near-optimal choice}
The reproducing kernel is a nearly optimal choice in \eqref{eq:general upper bound for f} when compared to the best value
\begin{equation*}
\min_{f\in\Pi_t}\,\frac{A\|f\|_{L^2(\S^d)}^2}{\min_{y\in S}|f(y)|^2} \geq \frac{A}{\dimp}\,.
\end{equation*}
To see this, fix $S\subseteq\S^d$ and let $f\in\Pi_t$ be arbitrary. Since $\min_{y\in S}|f(y)|\le |f(x)|$ for all $x\in S$ and 
$|f(x)|^2=|\langle f,k_x^t\rangle|^2\leq \|f\|_{L^2(\S^d)}^2 \dimp$, we always have
\begin{align*}
\frac{A\|f\|_{L^2(\S^d)}^2}{\min_{y\in S}|f(y)|^2}
&\geq 
\frac{A\|f\|_{L^2(\S^d)}^2}{|f(x)|^2}\ge
\frac{A}{\dimp}\,.
\end{align*}
By comparison, for small spherical caps $B_r(x)$ such that $r$ is correlated to $\nu\in(0,1)$, Theorem \ref{mutmeas} 
yields the bound $\frac{A}{\nu^2\,\dimp}$. In the context of the
estimate \eqref{eq:general upper bound for f}, the kernel $k^t_x$ (or
equivalently $f^t_x$) is therefore a nearly optimal choice for achieving the upper bound with respect to the spherical cap $B_r(x)$.

\subsection*{Sub-optimal choice}
Another natural choice for $f$ in \eqref{eq:general upper bound for f} is  
\begin{equation}\label{eq:Kt}
K^t(x,y)=K^t_x(y)=(x\cdot y)^t\,.
\end{equation}
It peaks at $y=x$ and is orthogonally invariant, just as $k^t$ is. Since the surface measure on $\S^d$ is rotation-invariant, the $L^2$-norm of $K^t_x$ does not depend on the specific choice of $x\in\S^d$ and we may define
\begin{equation}\label{eq:def Lambda}
\Lambda_{d,t}:=\|K^t_x\|_{L^2(\S^d)}^2 = \int_{\S^d} (x\cdot y)^{2t}\,dy\,.
\end{equation}
For all $y\in B_r(x)$, we have $x\cdot y\geq \cos r$ and therefore
$|K^t_x(y)|^2\geq (\cos r)^{2t}$. Hence, substituting  
$f= K^t_x$ in \eqref{eq:general upper bound for f}, leads to 
\begin{equation}\label{eq:up23}
\mu_t\big(B_r(x)\big) \leq \frac{A\, \Lambda_{d,t}}{(\cos r)^{2t}}\,.
\end{equation}

We still need to compare $\Lambda_{d,t}$ with the factor $1/\dimp$ in Theorem \ref{mutmeas}. Repeated application of the Laplace operator in spherical coordinates, as executed in \cite{BachocEhler} for instance, leads to
\begin{equation*}
\Lambda_{d,t} = \frac{1\cdot 3  \cdots
  (2t-1)}{(d+1)(d+3)\cdots(d+2t-1)}
=\frac{(\tfrac{1}{2})_t}{(\frac{d+1}{2})_t} \,,
\end{equation*}
where $(u)_t$ is the Pochhammer symbol $(u)_t=\frac{\Gamma(t+u)}{\Gamma(u)}=u(u+1)\cdots (u+t-1)$. Let us fix the dimension $d$, consider large $t$, and choose $r\asymp t^{-1}$, which matches the choice in \eqref{eq:r2}.  Since     
      \begin{equation*}
      \Lambda_{d,t}=\frac{\Gamma(t+\frac{1}{2})\Gamma(\frac{d+1}{2})}{\Gamma(\frac{1}{2})\Gamma(t+\frac{d+1}{2})} \sim \frac{\Gamma(\frac{d+1}{2})}{\Gamma(\frac{1}{2})}t^{-\frac{d}{2}}\,,
      \end{equation*}
whereas $ 1/\dimp\sim \frac{2}{d!}t^{-d}$, \eqref{eq:7} in Theorem \ref{mutmeas} is orders of magnitude sharper than \eqref{eq:up23}. This shows that the kernel $K^t$ is not sufficiently localized for the approach \eqref{eq:general upper bound for f} at scale $t^{-1}$. Nonetheless, its simplicity also has benefits that we will use next.

\subsection{Elementary consequences of the lower \mz\ inequality}
We now verify necessary support conditions for the lower \mz\ inequality to hold. Here, the simplicity of $K^{t}$ in \eqref{eq:Kt} is advantageous and it is therefore used to generate test functions in $\Pi_t$. 

Recall the moments $\Lambda_{d,t}= \int_{\S^d} (x\cdot y)^{2t}\,dy$ of the Lebesgue measure on $\S^d$ from \eqref{eq:def Lambda}. For $t=1$, we obtain $\Lambda_{d,1}=\frac{1}{d+1}$, and  $g(y)=(x\cdot y)^2$ leads to $\Lambda_{d,t}^{1/t}=\|g\|_{L^t(\S^d)}$. To estimate $\Lambda_{d,t}^{1/t}$, for $1\leq t \leq \infty$, the H\"older inequality yields
\begin{equation*}
\frac{1}{d+1}=\Lambda_{d,1}=\| g\|_{L^1(\S^d)} \leq \|g\|_{L^{t}(\S^d)}\leq \|g\|_{L^{\infty}(\S^d)} = 1\,.
\end{equation*}
Consequently, $\Lambda_{d,t}^{1/t}$ is sandwiched between 
\begin{equation}\label{eq:ineq for Lambda}
\frac{1}{d+1} \leq \Lambda_{d,t}^{1/t} \leq 1\,.
\end{equation}
and satisfies $\lim_{t\rightarrow\infty}\Lambda_{d,t}^{1/t} = 1$.

We derive the following elementary covering property of \mz\ measures. It implies that the support of \mz\ measures cannot be contained in a spherical cap that is too small.
\begin{lemma}\label{lemma:ally2}
Assume that $\mu_t$ satisfies the lower \mz\ inequality of order $t$ and let $r\in(0,\frac{\pi}{2}]$. If $\supp \mu_t \subseteq B_r(x)$, then 
\begin{equation}\label{eq:sin r}
(\sin r)^{2t} \geq  \frac{a}{\mu_t(\S^d)} \,\Lambda_{d,t}\,.
\end{equation}
\end{lemma}
In other words, if \eqref{eq:sin r} is not satisfied, then the support of $\mu_t$ cannot be contained in any ball of radius $r$. Also 
notice the requirement $0<r\leq \frac{\pi}{2}$, so that Lemma \ref{lemma:ally2} does not apply to the complementary cap  $\overline{\S^d\setminus B_r(x)} = B_{\pi-r}(-x)$. 
\begin{proof}
We first assume that $z=e_{d+1}$ is the north pole and $\supp \mu_t\subseteq B_{r}(e_{d+1})$. 
Consider $f(y)=y_1^t$, so that the lower \mz\ bound yields
\begin{equation}\label{eq:x12b}
a\,
\Lambda_{d,t} \leq  \int_{\S^d} |y_1|^{2t}\,d\mu_t(y)\,,\qquad \forall\, t\in\N\,.
\end{equation}

Since $y\in B_{r}(e_{d+1})$ implies $y_{d+1}\geq \cos r$ and  $\cos
r\geq 0$  for $0<r\leq \pi/2$, we  deduce $y_{d+1}^2\geq (\cos r)^2$ and, hence, 
\begin{equation*}
|y_1|^2\leq 1-|y_{d+1}|^2\leq 1-(\cos r)^2=(\sin r)^2\,.
\end{equation*}
Therefore, on $\supp\mu_t$ the function $|y_1|^{2t}$ is bounded by $(\sin r)^{2t}$, and we obtain
\begin{align*}
a \Lambda_{d,t}  \leq \int_{\S^d} |y_1|^{2t}\,d\mu_t(y) \leq (\sin r)^{2t} \mu_t(\S^d)\,.
\end{align*}
This directly translates into $(\sin r)^{2t} \geq \frac{a}{\mu_t(\S^d)} \Lambda_{d,t}$. 

For the general case $x\in\S^d$, we apply a rotation $O$ such that $Oe_{d+1} = x$. The push-forward measure $\tilde{\mu}_t:=O_*\mu_t$ is defined by $\tilde{\mu}_t(S)=\mu_t(O^*S)$ for every measurable subset $S\subseteq \S^d$. The condition $\supp \mu_t\subseteq B_r(x)$ directly implies $\supp \tilde{\mu}_t\subseteq B_r(e_{d+1})$. Since the space $\Pi_t$ is orthogonally invariant, $\tilde{\mu}_t$ also satisfies the lower \mz\ inequality with the same constant $a$. Our above analysis now implies the condition on $r$.
\end{proof}

\section{Bounds on the length of \mz\ curves}\label{sec:curves}
We specialize the upper volume bounds in Theorem \ref{mutmeas} to measures induced by curves and derive quantitative lower bounds on the length of \mz\ curves. 

A curve is a continuous, piecewise smooth function $\gamma:[0,L]\rightarrow\S^d$ that we may consider as a space curve in $\mathbb{R}^{d+1}$. We always assume that $\gamma$ is a closed curve, i.e., $\gamma(0)=\gamma(L)$, and parametrized by arc-length, so that the speed $\|\dot{\gamma}\|$ equals $1$ almost everywhere, the length satisfies $\ell(\gamma)=L$, and the line integral is 
\begin{equation*}
\int_{\gamma} f :=\int_0^L f(\gamma(u))\,du\,.
\end{equation*}

In the following we also consider a % slightly more general
weighted version. For a positive integer $t$, let $w_t$ be a nonnegative weight function that is defined and integrable on the trajectory of the curve $\gamma_t$. Every weighted curve $(\gamma_t,w_t)$ induces a measure $\mu_t$ satisfying $\mu_t(B_r(x))=\frac{1}{\ell(\gamma_t)}\int_{\gamma_t} 1_{B_r(x)} w_t$ and 
\begin{equation}\label{eq:measure from curve}
\int_{\S^d}|f(x)|^2\,d\mu_t=\frac{1}{\ell(\gamma_t)} \int_{\gamma_t} |f|^2 w_t\,.
\end{equation}
This enables us to define \mz\ curves. 
\begin{df}
We say that $(\gamma_t,w_t)$ is a (weighted) \mz\ curve (of order $t$ or for $\Pi_t$) if 
\begin{equation*}
a\|f\|^2_{L^2(\S^d)} \stackrel{\text{lower}}{\le} \frac{1}{\ell(\gamma_t)} \int_{\gamma_t} |f|^2 w_t \stackrel{\text{upper}}{\le} A\|f\|^2_{L^2(\S^d)}
\end{equation*}
holds for all $f\in\Pi_t$.
\end{df}
For related  concepts, see \cite{EG:2023, EGK:24,Lindblad1}.  
The upper volume bound in Theorem \ref{mutmeas} and the support
estimates in Lemma \ref{lemma:ally2} lead to the following lower bound
on the length of a \mz\ curve.
\begin{tm}[weighted curves]\label{tm:weighted curves}
Assume $t\geq d\geq 2$. If $(\gamma_t,w_t)$ is a weighted \mz\ curve of order $t$ on $\S^d$, then its length must satisfy
\begin{equation}\label{eq:in weight curve}
\ell(\gamma_t)\geq  \frac{\dimp}{t} \,\frac{\sqrt{d+2}}{4}\,\Big(\frac{a}{A}\Big)^{1+\frac{1}{2t}}\,.
\end{equation}
\end{tm}
Since the dimension $\dimp=\binom{d+t}{t} + \binom{d+t-1}{t-1}$ is  bounded by 
\begin{equation}\label{eq:estimate dimp}
\dimp  = \frac{(t+1)\cdots(t+d)}{d!}+\frac{ t \cdots(t+d-1)}{d!}\geq
\frac{2}{d!}t^d\, ,
\end{equation}
the main estimate  \eqref{eq:in weight curve} leads to $ \ell(\gamma_t) \gtrsim t^{d-1}$ with the quantitative version
\begin{equation}\label{eq:dim to td}
\ell(\gamma_t)\geq  \Big(\frac{a}{A}\Big)^{1+\frac{1}{2t}}\,\frac{\sqrt{d+2}}{2 \,d!}  \,t^{d-1}\,.
\end{equation}
As an important special case, we formulate the following concequence.
\begin{cor}\label{cor:designer}
Assume $t\geq d\geq 2$. If $\gamma_t$ is a $t$-design curve, then 
\[
    \ell(\gamma_t) \ge  \frac{\sqrt{d+2}}{2\, d!}  \,t^{d-1}\,.
\]
\end{cor}
\begin{proof}[Proof of Corollary \ref{cor:designer}]
Any $t$-design curve  has \mz\ constants $a = A = 1$, so that \eqref{eq:dim to td} provides the estimate.
\end{proof}
This was proved in \cite{EG:2023} for a generic constant with a
different method that requires exact quadrature. The novelty here is
the explicit constant that shows 
the dependence of the dimension.

\begin{proof}[Proof of Theorem \ref{tm:weighted curves}]
We will apply Theorem \ref{mutmeas}  and use the relation 
$r^2 =  (1-\nu)\frac{(d+2)}{t(t+d)}$ from \eqref{eq:r2}, with a value of $r=r(d,t,\nu)\in(0,\frac{\pi}{2}]$ to be determined by $\nu\in(0,1)$. 

\subsubsection*{Consequences of the upper volume bound:\\}
The upper volume bound in Theorem \ref{mutmeas} leads to
\begin{equation*}
\mu_t(B_{r}(x))=\frac{1}{\ell(\gamma_t)} \int_{\gamma_t} 1_{B_{r}(x)} w_t \leq \frac{A}{\nu ^2 \dimp }\,,
\end{equation*}
which yields the inequality
\begin{equation}\label{eq:low bound curve}
\ell(\gamma_t)\geq \frac{\nu ^2\, \dimp }{A}\int_{\gamma_t} 1_{B_{r}(x)}w_t\,.
\end{equation}
It remains to find an estimate on $\int_{\gamma_t} 1_{B_{r}(x)}w_t$ for some $x\in\S^d$.  
Next, we determine a radius $r$ by choosing $\nu$ such that $\gamma_t$ is not contained in $B_r(x)$.

\subsubsection*{Determine $r$ by choosing $\nu$ such that $\gamma_t\not\subseteq B_r(x)$:\\}
To ensure that $\gamma_t$ is not contained in any spherical cap $B_r(x)$, we must make sure that the radius $r=r(d,t,\nu)$ violates \eqref{eq:sin r} in Lemma \ref{lemma:ally2}, i.e., 
we need $(\sin r)^{2t} <  \frac{a}{\mu_t(\S^d)} \,\Lambda_{d,t}$,
where $\Lambda_{d,t}$ is as in \eqref{eq:def Lambda}. Since  $(\sin
r)^{2t} \leq r^{2t}$, $\Lambda_{d,t}\geq \big(\frac{1}{d+1}\big)^t$ by \eqref{eq:ineq for Lambda}, and the upper \mz\ inequality yields
$\mu_t(\S^d)\leq A$, it suffices to satisfy  
the slightly stronger inequality 
\begin{equation*}
r^{2t} <  \frac{a}{A}\,\Big(\frac{1}{d+1}\Big)^t\,.
\end{equation*}
This leads to the requirement 
 \begin{align*}
r^{2} =  (1-\nu) \frac{(d+2)}{\big(t(t+d)\big)}  < \Big(\frac{a}{A}\Big)^{1/t}\, \frac{1}{d+1}\,,
 \end{align*}
which is equivalent to 
 \begin{equation}\label{eq:not contained condition}
 1-\nu   < \Big(\frac{a}{A}\Big)^{1/t}\,\,\frac{t(t+d)}{(d+2)(d+1)}\,.
 \end{equation}
 For 
 $t\geq d\geq 2$, we obtain  
 \begin{equation} \label{nee1}
\frac{t(t+d)}{(d+2)(d+1)}
\geq \frac{2}{3}\,,
\end{equation}
 and we may now choose $\nu$ as 
\begin{equation*}
1-\nu =  \frac{1}{2}\Big(\frac{a}{A}\Big)^{1/t}\,.
\end{equation*}
For this value of $\nu $ we  know that $(\sin r)^{2t} <  \frac{a}{\mu_t(\S^d)}\, \Lambda_{d,t}$. Hence, $\gamma_t$ cannot be contained in any ball of radius $r$.

\subsubsection*{Find $x_t$ such that $\int_{\gamma_t} 1_{B_{r}(x_t)}w_t \geq 2r a$:\\}
Let $x$ lie on the trajectory of $\gamma_t$. Since $\gamma_t$ is a closed curve not contained in $B_r(x)$, it must enter and leave the ball, hence, at least a piece of length $2 r$ must be contained in $B_{r}(x)$. If $w_t\equiv 1$, then we are done, because $\int_{\gamma_t} 1_{B_{r}(x_t)}w_t \geq 2r$. If $w_t$ is not constant, however, it could be rather small on that particular segment and we cannot immediately deduce $\int_\gamma 1_{B_{r}(x)}w_t \geq 2r$.  
In the following we will argue that there must exist at least one $x_t$ on the trajectory of $\gamma_t$ such that $\int_{\gamma_t} 1_{B_{r}(x_t)}w_t$ is sufficiently big.
 
Recall that $L=\ell(\gamma_t)$ and let us extend the arc-length parametrization of the closed curve $\gamma_t$ periodically beyond the interval $[0,L]$. We also regard the characteristic function $1_{[-r,r]}$ as an $L$-periodic functions and now define $v:[0,L]\rightarrow\mathbb{R}$ by the cyclic convolution $v=(w_t\circ \gamma_t) * 1_{[-r,r]}$ given by
\begin{equation*}
v(u):=
\int_{0}^L w_t(\gamma_t(p))\, 1_{[-r,r]}(p-u)\,dp\,.
\end{equation*}
Integration over $[0,L]$ and $2r\leq L$ lead to
\begin{align*}
\int_0^L v(u)\,du 
& = \int_0^L \int_{0}^L w_t(\gamma_t(p))\, 1_{[-r,r]}(p-u)\,dp\,du
\\
& = \int_{0}^L w_t(\gamma_t(p)) \int_0^L \, 1_{[-r,r]}(p-u)\,du\,dp
\\
& = 2r \,\int_0^L w_t(\gamma_t(p))\,dp
\\
& = 2r \int_{\gamma_t}w_t \,.
\end{align*}
The lower \mz\ inequality implies $ \int_{\gamma_t} w_t\geq a\,L$, so
that we arrive at $\int_0^L v(u)\,du\geq 2raL$. Thus, there is
$u_t\in[0,L]$ such that $v(u_t)\geq 2r a$. We now select
$x_t:=\gamma_t(u_t)$. Since $\gamma_t(p)\in B_r(x_t)$ for
$p\in[u_t-r,u_t+r]$ and $w_t\geq 0$, we obtain 
\begin{equation}\label{eq:Ahlfors bound for curves}
\int_{\gamma_t} 1_{B_{r}(x_t)}w_t \geq  \int_{u_t-r}^{u_t+r} w_t(\gamma_t(p))\,dp = v(u_t)\geq 2r a\,.
\end{equation}

\subsubsection*{Final estimates:\\}
With the specified choice of $r$ and $\nu $ the lower  bound
\eqref{eq:low bound curve} turns into  
\begin{equation}\label{eq:low bound curve 2}
\ell(\gamma_t)\geq  \nu^2\, \dimp    \,\frac{a}{A}\,\,2r .
\end{equation}

Since $ \frac{a}{A}\leq 1$, we have
$\nu^2=\big(1-\frac{1}{2}\big(\frac{a}{A}\big)^{1/t}\big)^2\geq
\frac{1}{4}$ and  therefore 
\begin{align*}
r^2 &=  (1-\nu)\,\frac{d+2}{t(t+d)} = \frac{1}{t^2} (1-\nu)\,\frac{d+2}{1+d/t}\geq  \frac{1}{t^2} \frac{(1-\nu)}{2}\,(d+2)\,.
\end{align*}
We now insert $r$ into the bound \eqref{eq:low bound curve 2} and
substitute  $1-\nu=\frac{1}{2}\big(\frac{a}{A}\big)^{1/t}$ to get
\begin{align*}
\ell(\gamma_t) & \geq  \frac{1}{4}\frac{\dimp}{t} \,\frac{a}{A}\, 2\sqrt{\frac{(1-\nu)}{2}(d+2)} \\
& = \frac{1}{4} \frac{\dimp}{t} \,\frac{a}{A}\, \sqrt{d+2}\,\Big(\frac{a}{A}\Big)^{\frac{1}{2t}}\,,
\end{align*}
which concludes the proof.
\end{proof}
For the asymptotics we view the dimension  $d$ as fixed and let the
degree  $t$
grow. However, large  $t$  
is not essential for the bounds; the proof goes through with slightly
different constants, as is indicated by the extreme case $t=1$ of linear functions. 
\begin{rem}
If $(\gamma_1,w_1)$ is a weighted \mz\ curve of order $1$ on $\S^d$, then its length must satisfy
\begin{equation}\label{eq:in weight curve t=1}
\ell(\gamma_1)\geq  2\,\sqrt{d+1} \,\frac{d+1}{d+2}\,\Big(\frac{a}{A}\Big)^{3/2}\,.
\end{equation}
To see this, we may follow the proof of Theorem \ref{tm:weighted
  curves} with minor adjustments according to $t=1$. First, we have $\dim \Pi_1=d+2$ and define the radius $r$ as usual by $r^2=(1-\nu)\frac{(d+2)}{t(t+d)}=(1-\nu)\frac{d+2}{d+1}$. For an arbitrary parameter $\delta\in(0,1)$, we choose $1-\nu=
\delta\,\frac{a}{A}\,\frac{1}{d+2}$. Then \eqref{eq:not contained condition} is satisfied for $t=1$. Consequently, $\gamma_1$ is not contained in any ball of radius $r$, where $r^2=\delta\,\frac{a}{A} \,\frac{1}{d+1}$. 
Thus, the inequality \eqref{eq:low bound curve 2} holds, i.e., 
\begin{equation*}
\ell(\gamma_1)  \geq  \nu^2\, (d+2)   \,\frac{a}{A}\,\,2r\,.
\end{equation*}
Since  $\nu^2= (1-\delta\,\frac{a}{A}\,\frac{1}{d+2})^2 \geq
\big(\frac{d+1}{d+2}\big)^2$ and $r=\sqrt{\delta\,\frac{a}{A} \,\frac{1}{d+1}}$, we eventually derive
\begin{align*}
\ell(\gamma_1) & \geq  
\frac{(d+1)^{3/2}}{d+2}\,\Big(\frac{a}{A}\Big)^{3/2}\,\,2 \sqrt{\delta} \,.
\end{align*}
Since this inequality holds for all $\delta\in(0,1)$, we obtain \eqref{eq:in weight curve t=1}.
\end{rem}

\section{Sets of arbitrary Hausdorff dimension}\label{sec:Hausdorff}
While piecewise smooth  curves 
have Hausdorff dimension $1$, one may also consider measures that are
supported on 
$s$-dimensional subsets of the sphere. In this section we establish
lower bounds  on their Hausdorff measure for arbitrary dimension $s<d$.

\subsection{$s$-dimensional \mz\ sets}
For $0\leq s\leq d$, we denote the $s$-dimensional Hausdorff measure by $h_s$. 
Every $s$-dimensional Hausdorff measurable subset $X_t\subseteq \S^d$
with positive, finite Hausdorff measure  $0<h_s(X_t)<\infty$ and an
integrable weight function $w_t:X_t\rightarrow[0,\infty)$ induces a
measure $\mu_t$ on $\sd $ defined by  $\mu_t(B_r(z))=\frac{1}{h_{s}(X_t)}\int_{X_t \cap B_r(z)}w_t(x)\, dh_s(x)$ and 
\begin{equation*}
\int_{\S^d} |f(x)|^2\,d\mu_t(x) = \frac{1}{h_{s}(X_t)}\int_{X_t} |f(x)|^2 w_t(x) \,d h_{s}(x)\,.
\end{equation*}
This enables us to define \mz\ sets $X_t$.
\begin{df}
We say that an $s$-dimensional Hausdorff measurable subset $X_t\subseteq \S^d$ with weight function $w_t:X_t\rightarrow[0,\infty)$ satisfies the \mz\ inequalities (of order $t$ or for $\Pi_t$) if
\begin{equation*}
a \|f\|_{L^2(\S^d)}^2 \stackrel{\text{lower}}{\le} \frac{1}{h_{s}(X_t)}\int_{X_t} |f(x)|^2 w_t(x) \,d h_{s}(x) \stackrel{\text{upper}}{\le} A \|f\|_{L^2(\S^d)}^2
\end{equation*}
holds for all $f\in\Pi_t$.
\end{df}

For such \mz\ sets we can prove  an extension of
Theorem~\ref{tm:weighted curves}  to the whole range of Hausdorff
dimensions  $0\leq s\leq d$. 
 
\begin{tm}[$s$-dimensional sets]\label{tm:s-dim}
Assume that $X_t\subseteq \S^d$ is an $s$-dimensional Hausdorff
measurable subset with weight function $w_t\geq 0$ satisfying the
upper \mz\ inequality of order $t$. If there exists a point $x_t\in
X_t$, parameters  $\nu\in(0,1)$ and $q>0$,   such that 
\begin{equation}\label{eq:Ahlfors3}
\int_{X_t\cap B_{r}(x_t)}w_t(x) \,dh_s(x) \geq q\, r^s\,,\qquad r^2 =  (1-\nu)\frac{(d+2)}{t(t+d)}\,,
\end{equation}
then the Hausdorff volume of $X_t$ satisfies
\begin{equation*}
h_{s}(X_t) \geq c \, \frac{\dimp}{t^{s}}\,,
\end{equation*}
with the ``constant''  $c=c(d,t,q,\nu,A,s)= \frac{q\, \nu^2}{A}\big(\frac{1-\nu}{1+d/t}\,(d+2)\big)^{s/2}$. 
\end{tm}
In case $t=1$, we have $c=\frac{q\, \nu^2}{A}\big((1-\nu)\frac{d+2}{d+1}\big)^{s/2}$. 
For the standard case $t\geq d$, we can replace $c$ by its lower
estimate $c\geq \frac{q\,
  \nu^2}{A}\big(\frac{1-\nu}{2}\,(d+2)\big)^{s/2}$.
 
\begin{proof}[Proof of Theorem \ref{tm:s-dim}]
The radius $r$ is as needed in Theorem \ref{mutmeas} when applied to $\mu_t$, and we obtain
\begin{equation*}
\mu_t(B_{r}(x_t))=\frac{1}{h_s(X_t)} \int_{X_t\cap B_{r}(x_t)} w_t(x)\,d h_s(x)\leq \frac{A}{\nu ^2 \dimp }\,.
\end{equation*}
The assumption \eqref{eq:Ahlfors3} leads to 
\begin{align*}
h_{s}(X_t) & \geq \int_{X_t\cap B_{r}(x_t)} w_t(x)\,d h_s(x)\,\frac{\nu ^2 \dimp}{A}\\
& \geq q \,r^{s}\,\frac{\nu ^2 \dimp}{A}\,.
\end{align*}
Since $r= \frac{1}{t}\sqrt{(1-\nu)\frac{d+2}{1+d/t}}$,   
we may define the constant $c:=\frac{q\, \nu^2}{A}\big(\frac{1-\nu}{1+d/t}\,(d+2)\big)^{s/2}$ and deduce 
\begin{align*}
h_{s}(X_t)& \geq  c\,  \frac{\dimp}{t^{s}}\,,
\end{align*}
which concludes the proof.
\end{proof}

\begin{rem}
  The technical condition~\eqref{eq:Ahlfors3} is related to the notion
  of Ahlfors regularity from geometric measure theory. 
A Hausdorff measurable subset $X\subseteq\S^d$ with weight function $w:X\rightarrow[0,\infty)$ is called lower Ahlfors $s$-regular if there is a constant $q>0$ such that 
\begin{equation}\label{eq:Ahlfors}
\int_{X\cap B_r(x)}w(y)\,dh_s(y) \geq q\, r^s\,,\qquad \forall \, x\in X\,,\quad 0<r<\diam(X)\,.
\end{equation}
Note, however, that  we need  this condition only at  a single point
$x_t\in X_t$ and for a single  radius $r$.  Since $q,\nu ,r$ are not independent, one could further optimize the
constant $c$ subject to $qr^s= \mathrm{const}$ and to the relation
between $r$ and $\nu $. 
\end{rem}
The   regularity is automatically satisfied for point sets
 and curves.
\begin{rem}[points]\label{rem:point zero}
For a finite point set $X_t$, the maximum $w_{\max}=\max_{x\in X_t} w_t(x)$ of the weights is attained at $x_{\max}\in X_t$ such that $w_{\max}=w_t(x_{\max})$. Then $X_t$ with weight function $w_t$ is lower Ahlfors $0$-regular at $x_{\max}$ for $q=w_{\max}$ and all $r\in(0,\pi]$. 
\end{rem}

\begin{rem}[curves]\label{rem:curve and Alhfors}
 For $s=1$, the
Hausdorff measure $h_1$ corresponds to the arc-length and
$\frac{\dimp}{t}\asymp t^{d-1}$. This  is consistent with the results
for \mz\ curves in Theorem \ref{tm:weighted curves}.

In light of Section~\ref{sec:curves}, the regularity assumption follows for curves directly from the lower \mz\ inequality. Indeed, the proof of Theorem \ref{tm:weighted curves} shows that the lower \mz\ inequality forces the weighted curve to be lower Ahlfors $1$-regular at least at some point $x_t$ for $q = 2a$, cf.~\eqref{eq:Ahlfors bound for curves}, with the radius $r$ specified in~\eqref{eq:r2} of Theorem \ref{mutmeas}. 
\end{rem}

\subsection{\mz\ points}
Let us now consider the case $s=0$ more closely. If a finite subset $X_t\subseteq \S^d$ with weight function $w_t$ satisfies the \emph{lower} \mz\ inequality of order $t$, then an elementary and well-known 
dimension count yields that the cardinality of $X_t$ must obey
\begin{equation}\label{eq:origin}
\#X_t \ge \dimp\,.
\end{equation}
The quotient $\frac{\# X_t}{\dimp}$ is usually referred to as the oversampling factor. 

Typically, the \emph{lower} \mz\ inequalities are widely viewed as the
more substantial part, so that the \emph{upper} \mz\ bound often
receives comparatively little attention in the literature. In contrast
to $\mathbb{R}^d$ or other non-compact manifolds, the sphere $\S^d$ is compact and hence a confined space, in which the local
density of \emph{upper} \mz\ points cannot exceed the (scaled) oversampling
factor $\frac{\# X_t}{\dimp} \,A$. This also provides a lower bound on the cardinality. 
\begin{cor}\label{cor:finito}
If a finite subset $X_t\subseteq\S^d$ with weight function $w_t:X_t\rightarrow[0,\infty)$ satisfies the upper \mz\ inequality of order $t$, then
\begin{equation*}
w_{\max}=\max_{x\in X_t}w_t(x)\leq \frac{\#X_t}{\dimp}\,A\, .
\end{equation*}
 Equivalently, it must hold
\begin{equation}\label{eq:low numbers}
\#X_t\geq \dimp\,\frac{w_{\max}}{A}\,.
\end{equation}
\end{cor}
Thus, we have a lower bound on the cardinality comparable to
\eqref{eq:origin}. It is remarkable  that we need  not  assume  the lower \mz\ inequality. 

\begin{proof}
As mentioned in Remark \ref{rem:point zero}, the Ahlfors regularity assumption in the point $x_{\max}$ with $w_{\max}=w_t(x_{\max})$ is always satisfied for $q=w_{\max}$ and all $r\in(0,\pi]$. Thus, Theorem \ref{tm:s-dim} leads to 
\begin{equation*}
\# X_t\geq \frac{w_{\max}\,\nu^2}{A} \dimp\,,
\end{equation*}
for all $\nu\in(0,1)$. The inequality must then also be valid for $\nu=1$, so that 
\begin{equation*}
w_{\max}\leq \frac{\#X_t}{\dimp}\,A
\end{equation*}
must hold. 
\end{proof}

\section{Lower volume bounds}
This section complements the upper volume bounds of Theorem
\ref{mutmeas} by (almost) matching lower bounds. The arguments are
more involved and  we will use  precise
off-diagonal decay estimates of the kernel. 

As a preparation, we notice that the integral over a
rotation-invariant function on $\sd $ can be reduced to a
one-dimensional integral via the formula~\cite{Reimer}
\begin{equation}
  \label{eq:oc1}
  \int _{\sd } f(x\cdot y ) \, dy = \frac{\vol(\mathbb{S}^{d-1})}{\vol(\sd )}
  \int _0^\pi f(\cos \theta ) \, (\sin  \theta )^{d-1} \, d\theta  \, ,
\end{equation}
where $\vol(\mathbb{S}^{d-1}) := \frac{2\pi ^{d/2}}{\Gamma (d/2)}$
denotes the Euclidean non-normalized surface measure of the full
$d-1$-sphere.  To keep track of the many constants, we abbreviate the
coefficient in front of the integral by 
\begin{equation}\label{eq:theta def}
\Theta_{d}:=\frac{\vol(\mathbb{S}^{d-1})}{\vol(\sd )} =
\frac{\Gamma(\frac{d+1}{2})}{\sqrt{\pi}\Gamma(\frac{d}{2})}\,.
\end{equation}
The integral over a spherical cap can then be written as
\begin{equation}\label{eq:Br integrals}
  \int _{B_r (x) } f(x\cdot y ) \, dy = \Theta_d
  \int _0^r f(\cos \theta ) \, (\sin  \theta )^{d-1} \, d\theta  \, .
\end{equation}
In particular,  the volume of a spherical cap is
\begin{equation*}
|B_r(x)| = \Theta_d \int_0^r (\sin \theta)^{d-1}\,d\theta\,.
\end{equation*}
Since $\sin\theta\geq \theta - \theta^3/6\geq \theta(1-\frac{r^2}{6})$, for $0\leq \theta\leq r$, this leads to the elementary lower estimate
\begin{equation}\label{eq:low Br}
|B_r(x)| \geq \Theta_d \frac{\big(1-\tfrac{r^2}{6}\big)^{d-1}}{d}\,r^d
\end{equation}
that we use later.

\subsection{Off-diagonal decay of the reproducing kernel}\label{sec:prep}
We next complement Proposition \ref{cor:conc} on  the local
concentration of the reproducing kernel  by a result about  its off-diagonal decay.

In the proofs  we will use very strong, uniform estimates for Jacobi polynomials
from \cite{EMN}, see also \cite{Krasikov}. These estimates are
formulated for  \emph{orthonormal}  Jacobi 
polynomials and contain the  normalization factor   
\begin{equation*}
\kappa^{(\alpha,\beta)}(t) =
\frac{2^{\alpha+\beta+1}}{2t+\alpha+\beta+1}\,
\frac{\binom{t+\alpha}{t}}{\binom{t+\alpha+\beta}{t+\beta}}\, ,
\end{equation*}
so that the polynomials
\begin{equation*}
\frac{1}{\sqrt{\kappa^{(\alpha,\beta)}(t)}}P_t^{(\alpha,\beta)}
\end{equation*}
are  orthonormal on $[-1,1]$ with respect to the weight function
$(1-x)^{\alpha}(1+x)^{\beta}$.  
The \EMN\ bound \cite{EMN} for $P^{(\alpha,\beta )}_t$ with $\alpha,
\beta > -1/2$ is then
given by
\begin{equation}
  \label{eq:oc5}
(1-u^2)^{1/2} (1-u)^\alpha (1+u)^\beta
\frac{1}{\kappa^{(\alpha,\beta)}(t)}\, \big|P_t^{\alpha,\beta}(u)\big|^2 \leq
\frac{2e}{\pi}(2+\sqrt{\alpha ^2 + \beta ^2}) \, .
\end{equation}
For $(\alpha,\beta)=(\frac{d}{2},\frac{d}{2}-1)$ we abbreviate the
constant on  the right-hand side as 
\begin{equation}\label{eq:Md}
M_d=\frac{2e}{\pi}\Big(2+\sqrt{\frac{d^2}{2}-d+1}\Big)
\end{equation}
and note that  $M_d\leq \frac{e}{\pi}(4+\sqrt{2}\,d)$ and $M_d \asymp d$. 

We additionally introduce the constant  
\begin{equation}\label{eq:Omegad0}
\Omega_{d,t}:=\frac{2\,\kappa^{(\frac{d}{2},\frac{d}{2}-1)}(t)}{\binom{t+\frac{d}{2}}{t}^2}
\end{equation}
and recall $\Theta_{d}=
\frac{\Gamma(\frac{d+1}{2})}{\sqrt{\pi}\Gamma(\frac{d}{2})}\asymp
\sqrt{d}$ from \eqref{eq:theta def}. With all these constants the
off-diagonal decay of the reproducing kernel can be described as
follows. 
 
\begin{lemma}\label{lemma:asymptotics Darboux}
For all $0<r\leq \pi$, we have 
\begin{equation}  \label{eq:oc7}
\int_{\S^d\setminus B_{r}(x)} \big|\widetilde{P}_t^{(\frac{d}{2},\frac{d}{2}-1)}(x\cdot y)\big|^2\,dy \leq 
M_d\,\Theta_d \, \frac{ \Omega_{d,t}}{r}\,.
\end{equation}
\end{lemma}
In the proof we will find 
\begin{equation}\label{eq:Omegad}
\Omega_{d,t}
=\frac{2^{{d+1}}}{\binom{d}{\frac{d}{2}}}\,\frac{1}{(t+\frac{d}{2})\,\dimp}\,,
\end{equation}
so that 
$\Omega_{d,t} = \mathcal{O}(t^{-(d+1)})$. Thus, the mass of the
normalized Jacobi kernel outside a spherical cap of radius $r$ is of
order  $t^{-(d+1)}$.
\begin{proof}
The \EMN\ bound \eqref{eq:oc5} implies for the standard Jacobi
polynomial $P_t^{(\alpha,\beta)}$ that 
\begin{align*}
\left|P_t^{(\alpha,\beta)}(u)\right|^2&\leq \kappa^{(\alpha,\beta)}(t)\,M_d (1-u)^{-\alpha-\frac{1}{2}}(1+u)^{-\beta-\frac{1}{2}}\\
& = \kappa^{(\alpha,\beta)}(t)\,M_d (1-u)^{\beta-\alpha}(1-u^2)^{-\beta-\frac{1}{2}}\,.
\end{align*}
For the normalized Jacobi polynomial $\widetilde{P}^{(\alpha,\beta)}_t$ and $x\cdot y = \cos \phi $, this leads to
\begin{align*}
\big|\widetilde{P}^{(\alpha,\beta)}_t(x\cdot y)\big|^2&=\frac{1}{\binom{t+\alpha}{t}^2} \big|P_t^{(\alpha,\beta)}(\cos\phi)\big|^2\\
&\leq M_d\,\frac{\kappa^{(\alpha,\beta)}(t) }{\binom{t+\alpha}{t}^2} \, (1-\cos\phi)^{\beta-\alpha}(\sin\phi)^{-2\beta-1}\,.
\end{align*}
Note that $y\in \S^d\setminus B_{r}(x)$ translates
into $r<\phi<\pi $.  By
switching the integration from $y$ to $\phi$ and using
\eqref{eq:Br integrals}, 
we  obtain  
\begin{align*}
\int_{\S^d\setminus B_{r}(x)} \big|\widetilde{P}_t^{(\alpha,\beta)}(x\cdot y)\big|^2\,dy 
& \leq  M_d\,\Theta_d\,\frac{\kappa^{(\alpha,\beta)}(t) }{\binom{t+\alpha}{t}^2}  \int_{r}^{\pi} (1-\cos\phi)^{\beta-\alpha}(\sin \phi)^{-2\beta-1}(\sin\phi)^{d-1} \,d\phi\,.
\end{align*}
For $(\alpha,\beta)=(\frac{d}{2},\frac{d}{2}-1)$, the sines cancel each other, so that the integral reduces to
\begin{align*}
\int_{r}^{\pi} (1-\cos\phi)^{-1} \,d\phi = \tfrac{1}{2} \int_{r}^{\pi } (\sin\tfrac{\phi}{2})^{-2} \,d\phi
& = \left(-\cot(\tfrac{\pi}{2}) +(\cot\tfrac{r}{2})\right) =\cot\tfrac{r}{2}
 \leq 2/r\,.
\end{align*}
The factor of interest is 
\begin{align*}
\frac{\kappa^{(\alpha,\beta}(t)}{\binom{t+\alpha}{t}^2}
& =
      \frac{2^{\alpha+\beta+1}}{2t+\alpha+\beta+1}
                                                \,\, 
                                          \frac{1}{\binom{t+\alpha}{t}\,\binom{t+\alpha+\beta}{t+\beta}}\, ,
\end{align*}
and, for $(\alpha,\beta)=(\frac{d}{2},\frac{d}{2}-1)$ and after some
manipulations of binomial coefficients, it  reduces to  
\begin{align*}
\tilde{\Omega}_{d,t}:=\frac{\kappa^{(\frac{d}{2},\frac{d}{2}-1)}(t)}{\binom{t+\frac{d}{2}}{t}^2} & =
\frac{2^{d}}{(2t+d)\, \binom{t+\frac{d}{2}}{t}\,\binom{t+d-1}{t+\frac{d}{2}-1}}\\
&= \frac{2^{d}}{(2t+d)}\,\frac{(t+d)}{(t+\frac{d}{2})}\, \frac{1}{\binom{t+d}{t}\,\binom{d}{\frac{d}{2}}}\\
& = \frac{2^{{d}}}{\binom{d}{\frac{d}{2}}}\,\frac{1}{(t+\frac{d}{2})\,\dimp}\,,
\end{align*}
where the latter equality is due to $\dimp = \binom{t+d}{t}\,\frac{2t+d}{t+d}$ as stated in \eqref{eq:dimp}.

Using the abbreviations for the three constants, we obtain
\begin{equation*}
\int_{\S^d\setminus B_{r}(x)} \big|\widetilde{P}_t^{(\frac{d}{2},\frac{d}{2}-1)}(x\cdot y)\big|^2\,dy
 \leq M_d  \,\Theta_d \,\tilde{\Omega}_{d,t}\, \frac{2}{r} \, .
\end{equation*}
Finally, set $\Omega_{d,t}=2\tilde{\Omega}_{d,t}$.
\end{proof}

\subsection{Quantitative lower volume bounds}\label{sec:lower volume bounds}
We now verify that \mz\ measures must have enough mass in any spherical cap of sufficient size. The proof relies on the off-diagonal decay estimates in Lemma \ref{lemma:asymptotics Darboux} and also makes use of the upper volume bounds in Theorem \ref{mutmeas}.
\begin{tm}[lower volume bounds]\label{main}
For $\epsilon \in(0,1)$, we define the constant
\begin{equation}\label{eq:Vd}
V_{d}=
\frac{A/a}{(1-\epsilon)} \,\pi \,M_d\,d^2 \,2^{\frac{d}{2}+6}(\sqrt{d}/e)^{d} \,,
\end{equation}
where $M_d$ is the \EMN\ bound as in \eqref{eq:Md}. 
If $\mu_t$ is a \mz\ measure of order $t$ on $\S^d$ and 
\begin{equation*}
R \geq  \frac{V_{d}}{t}\,,
\end{equation*}
then $\mu_t$ satisfies the lower volume bound
  $$
    \mut \big( B_{R} (z) )
    \geq \frac{a\,\epsilon}{\dimp} \,, \qquad
    \forall\, z\in \sd \, .
$$
\end{tm}
Theorem \ref{main} complements the upper volume bounds in Theorem \ref{mutmeas}. In particular for the radius $R\asymp t^{-1}$, the lower volume bound scales like $t^{-d}$, in agreement with the complementary upper bound in Theorem \ref{mutmeas}.

\begin{rem}[Requires $t\gg d$]
To ensure $R<\pi$, the term $t$ must counterbalance $V_{d}$. Yet,
$V_{d}$ grows rapidly with $d$, in fact, super-exponentially in $d$. 
Consequently, $t$ has to be large relative to $d$, otherwise the lower bound becomes trivial since every cap of radius at least $\pi$ contains the whole sphere.
\end{rem}

The proof of Theorem \ref{main} requires some preparations. We begin by applying the 
\mz\ inequalities to the normalized Jacobi polynomial $f^t_y \in \Pi_t$,
\begin{equation}
  \label{eq:3}
  f^t_y(x) = \widetilde{P}^{(\frac{d}{2},\,\frac{d}{2}-1)}_t(x \cdot y)\,.
\end{equation}
The \mz\ inequality then yields
\begin{equation}
  \label{eq:4}
  \frac{a}{\dimp} \leq \int_{\sd} |f^t_y(x)|^2 \, d\mu_t(x) \leq \frac{A}{\dimp}\,.
\end{equation}
We now integrate the parameter $y$ over a  small ball $B_{r}(z)$ for
some $r>0$ to be determined and obtain 
\begin{equation}
  \label{eq:5}
\frac{a|B_{r}|}{\dimp}   \leq  \int _{B_{r}
    (z)}\int _{\sd } |f\ty (x)|^2 \, d\mut (x) \, dy \leq
\frac{A|B_{r}|}{\dimp}\,.
\end{equation}

For $r\le R\le \pi$, we decompose the integration in the center variable $x$ into
$\S^d = B_R(z)\cup \big(\S^d\setminus B_{R}(z)\big)$ and define
\begin{align}
  I_{r,R}(z) &:= \int_{B_r(z)}\int_{B_R(z)} |f\ty(x)|^2 \, d\mut(x)\,dy,\,\\
  I\!I_{r,R}(z) &:= \int_{B_r(z)}\int_{\S^d\setminus B_{R}(z)} |f\ty(x)|^2 \, d\mut(x)\,dy\,.
  \label{eq:9}
\end{align}
 The above decomposition into near and far regions is  a first step
 towards a quantitative lower estimate of $\mu_t(B_R(z))$ by suitably estimating $I_{r,R}$ and rearranging terms.
\begin{lemma}\label{lemma:first step}
If $\mu_t$ satisfies the \emph{lower} \mz\ inequality of order $t$ on $\S^d$, then for every $r\leq R\le \pi$,
\[
  \mu_t\big(B_R(z)\big)\;\ge\; \frac{a}{\dimp}\;-\;\frac{I\!I_{r,R}(z)}{|B_r|}\,.
\]
\end{lemma}

\begin{proof}
Since $|f\ty (x)|^2\leq 1$, we obtain
\begin{equation*}
I_{r,R}(z) \leq  \mu_t(B_{R}(z)) \,|B_{r}
    (z)|\,.
\end{equation*}

The lower bound in \eqref{eq:5} leads to 
\begin{align*}
\frac{a |B_{r}|}{\dimp} -I\!I_{r,R}(z) &\leq I_{r,R}(z)  
\leq \mu_t(B_{R}(z)) \,  |B_{r}|\, , 
\end{align*}
so that division by $|B_r|$ yields the claim.
\end{proof}
According to Lemma \ref{lemma:first step}, we now need 
a suitable upper bound on $I\!I_{r,R}$. 
\begin{lemma}\label{lemma:switch int}
Assume that $\nu \in(0,1)$ and $r>0$ are related by  $r^2 =  (1-\nu)\frac{(d+2)}{t(t+d)}$ as in \eqref{eq:r2} and let $r\leq R\leq\pi $. 
If $\mu_t$ satisfies the \emph{upper} \mz\ inequality of order $t$ on $\S^d$, then 
\begin{equation*}
I\!I_{r,R}(z) \leq \frac{A}{\nu ^2 \dimp } 
\int_{\S^d\setminus B_{R-r}(z)} |f^t_z(x)|^2 \, dx\,.
\end{equation*}
\end{lemma}
\begin{proof}
We first rewrite $I\!I_{r,R}(z)$ in \eqref{eq:9}. The integral over $y$ is treated first and the rewrite relies on applying a suitable reflection. 

For $v\in\S^d$, the Householder transform
$U=I-2vv^\top\in\mathcal{O}(d+1)$ is the reflection with respect to
the hyperplane orthogonal to $v$. Therefore, given $x,z\in\S^d$ and choosing $v=(z-x)/\|z-x\|$, the reflection satisfies $Uz=U^\top z=x$. The orthogonal invariance of
 $f^t(y,x)=\widetilde{P}^{(\frac{d}{2},\,\frac{d}{2}-1)}_t(x \cdot y)$
 then leads to 
 $$
 f^t(U y,x) = f^t(y, U^\top x) = f^t( y, z)\,, 
 $$
and $U^\top B_{r}  (z)=B_{r}  (U^\top z)=B_{r} (x)$ 
implies
\begin{align*} 
 \int _{B_{r}  (z)} |f^t(y, x)|^2 \, dy &=   \int _{U^\top B_{r}  (z)} |f^t(Uy, x)|^2 \, dy \\
 &=   \int _{B_{r}  (U^\top z)} |f^t(y, z)|^2 \, dy \\
  &=   \int _{B_{r}  (x)} |f^t(y, z)|^2 \, dy \,.
\end{align*}
The identity 
$1 _{B_{r}  (x)}(y) = 1_{B_{r} (y)}(x)$ yields the full expression for $I\!I_{r,R}$ as
\begin{align*}
  I\!I_{r,R}(z) 
  & = \int _{\S^d\setminus B_{R}(z)} \int _{B_{r }(x)}|f^t(y,z)|^2  \, dy  \, d\mut (x)  \\
  & =\int _{\S^d\setminus B_{R}(z)} \int_{\S^d}|f^t(y,z)|^2  1 _{B_{r}(x)}(y)\, dy  \, d\mut (x)  \\
  & =\int _{\S^d\setminus B_{R}(z)} \int_{\S^d}|f^t(y,z)|^2  1_{B_{r} (y)}(x)\, dy  \, d\mut (x)  \\
   & =\int_{\S^d} |f^t(y,z)|^2\int _{\S^d\setminus B_{R}(z)}   1_{B_{r } (y)}(x) \, d\mut (x)\, dy  \,.
\end{align*}
 If $B_{r}(y) \cap \big(\S^d\setminus B_{R}(z)\big) \neq \emptyset $,
 then $y \in  \S^d\setminus B_{R-r}(z)$. 
 Therefore, the integral over $x$ is majorized by
 \begin{align*}
   \int _{\S^d\setminus B_{R}(z)} \, 1_{B_{r}
   (y)}(x)  \, d\mut (x) &\leq 1_{\S^d\setminus B_{R-r}(z)}(y) \int _{\sd } \, 1_{B_{r}
   (y)}(x)  \, d\mut (x) \\
&= 1_{\S^d\setminus B_{R-r}(z)}(y)\, \mut (B_{r} (y))  \, .
 \end{align*}
Inserting in $I\!I_{r,R}$  and applying Theorem~\ref{mutmeas} to estimate
$\mut (B_{r} (y))$, we have now for any given  $\nu\in(0,1)$, such that $r\leq R\leq \pi$, 
\begin{align*}
  I\!I_{r,R}(z) &\leq \int_{\S^d} |f^t(y,z)|^2  1_{\S^d\setminus B_{R-r}(z)}(y)\, \mut (B_{r} (y))\, dy\\
&\leq \frac{A}{\nu ^2 \dimp } \int
  _{\S^d\setminus B_{R-r}(z)} |f^t(y,z)|^2 \, dy \, .
\end{align*}
This concludes the proof.
\end{proof}
To verify Theorem \ref{main}, we now suitably estimate the integral 
$\int_{\S^d\setminus B_{R-r}(z)} |f^t_z(x)|^2 \, dx$ by the off-diagonal decay estimates in Lemma \ref{lemma:asymptotics Darboux}.

\begin{proof}[Proof of Theorem \ref{main}]
The proof comes down to determining $r$ and $R$ such that the requirement
\begin{equation}\label{eq:bound on IIR}
\frac{I\!I_{r,R}}{|B_{r}|} \leq  \frac{a}{\dimp}(1-\epsilon),
\end{equation}
holds, because  Lemma \ref{lemma:first step} then  implies 
\begin{equation*}
 \mu_t\big(B_R(z)\big) \,\geq \, \frac{a}{\dimp} -\frac{I\!I_{r,R}(z)}{ |B_{r}|}\geq \epsilon \frac{a}{\dimp}\,.
\end{equation*}

To verify \eqref{eq:bound on IIR}, we use  Lemma \ref{lemma:switch
  int} and 
the off-diagonal decay \eqref{eq:oc7} of the kernel.  
    For $0<R-r<\pi$ and $r^2 =  (1-\nu)\frac{(d+2)}{t(t+d)}$, we have 
\begin{equation*}
\int_{\S^d\setminus B_{R-r}(z)} \big|f^t_z(x)\big|^2\,dx
 \leq \,\frac{M_d\,\Theta_d \, \Omega_{d,t}}{R-r}\, ,
\end{equation*}
and Lemma \ref{lemma:switch int} leads to
\begin{equation*}
I\!I_{r,R} \leq \frac{A}{\nu ^2\,\dimp} \,\frac{ M_d\, \Theta_d \,\Omega_{d,t}}{R-r}\,.
\end{equation*}

By taking into account $|B_r(x)| \geq \Theta_d \frac{\big(1-\tfrac{r^2}{6}\big)^{d-1}}{d}\,r^d$ as derived in \eqref{eq:low Br}, choose $R$ according to the stronger requirement 
\begin{equation*}
\frac{I\!I_{r,R}}{|B_{r}|} \leq \frac{A}{\nu ^2\,\dimp} \,\frac{ M_d\,\Theta_d \, \Omega_{d,t}}{R-r}\, \frac{d}{\Theta_d \big(1-\tfrac{r^2}{6}\big)^{d-1}\,r^d} \,  \leq \frac{a(1-\epsilon)}{\dimp}\,.
\end{equation*}
This simplifies to 
\begin{equation*}
\frac{A\,M_d\, \Omega_{d,t}\,d}{\nu ^2
  \,(R-r)\,\big(1-\tfrac{r^2}{6}\big)^{d-1}\,r^d} \,  \leq
a(1-\epsilon) \, ,
\end{equation*}
and solving for  $R-r$ yields
\begin{equation}\label{eq:RR}
\frac{A\,d\,M_d}{\nu ^2 \,a(1-\epsilon)\,\big(1-\tfrac{r^2}{6}\big)^{d-1}} \,\frac{\Omega_{d,t}}{r^d}  \leq R-r\,. 
\end{equation}
As already argued above,  \eqref{eq:RR} implies $\mut \big( B_{R} (z) )
    \geq \frac{a\,\epsilon}{\dimp}$. 
    
It remains to simplify the left-hand side of \eqref{eq:RR},  to remove
the dependence on $t$, and to find an explicit  
constant by a suitable choice of $\nu $ and $r$.

\subsubsection*{Determine $\nu$, $r$, $(t\,r)^d$, and estimate  $\big(1-\tfrac{r^2}{6}\big)^{d-1}$:\\}
For $t\geq d+1$, we choose $\nu$ according to $(1-\nu)=\frac{1}{2}(1+d/t)$. The constraint $t\geq 2d$ then leads to
\begin{equation*}
\nu^2 = \Big(1-\tfrac{1}{2}(1+d/t)\Big)^2 \geq \Big(\frac{1}{4} \Big)^2 = 2^{-4}\,.
\end{equation*}
For this choice of $\nu$, the radius $r$ is 
\begin{equation*}
r=\frac{1}{t}\sqrt{(1-\nu)\frac{d+2}{1+d/t}}=\frac{1}{t}\sqrt{\frac{d+2}{2}}\,.
\end{equation*}
Hence, we observe $(t\,r)^d = \big(\frac{d}{2}+1\big)^{d/2}$. 

To get a lower bound on $\big(1-\tfrac{r^2}{6}\big)^{d-1}$, we use $t\geq d+2$ and majorize $r^2$ by 
\begin{equation*}
r^2=\frac{1}{t^2} \,\frac{d+2}{2} \leq \frac{1}{2(d+2)}\leq \frac{1}{2d}\,.
\end{equation*}
Hence, we have 
\begin{equation*}
\big(1-\tfrac{r^2}{6}\big)^{d-1} \geq \big(1-\tfrac{1}{12 d}\big)^{d-1} \geq e^{-\frac{1}{12}}\,.
\end{equation*}

\subsubsection*{Estimate $\frac{\Omega_{d,t}}{r^d}$:\\}
We now consider $\frac{\Omega_{d,t}}{r^d}=\frac{\Omega_{d,t}t^d}{(t\,r)^d}$, where $\Omega_{d,t}$ is given in \eqref{eq:Omegad}. 
To estimate the numerator  $\Omega_{d,t}t^d$, we use $\dimp\geq \frac{2}{d!}t^d$ from \eqref{eq:estimate dimp} and obtain
\begin{equation*}
\Omega_{d,t}\,t^d  = \frac{2^{{d+1}}}{\binom{d}{\frac{d}{2}}}\,\frac{1}{(t+\frac{d}{2})}\,\frac{t^d}{\dimp} 
\leq  \frac{2^{{d}} \,\Gamma(\frac{d}{2}+1)^2}{(t+\frac{d}{2})}\,.
\end{equation*}
By combining the above with $(t\,r)^d = \big(\frac{d}{2}+1\big)^{d/2}$, we have verified that 
\begin{equation*}
\frac{\Omega_{d,t}}{r^d}=\frac{\Omega_{d,t}t^d}{(t\,r)^d} \leq \frac{2^{d} \,\Gamma(\tfrac{d}{2}+1)^2}{\big(\frac{d}{2}+1\big)^{d/2}}\,\frac{1}{t+\frac{d}{2}}\,.
\end{equation*}

\subsubsection*{Collect all estimates:\\}
We now collect all the estimates for the terms in \eqref{eq:RR}, use
$\frac{1}{t+\frac{d}{2}}\leq \frac{1}{t}$, and choose $R$ so that 
\begin{equation*}
\left(\frac{A\,e^{\frac{1}{12}}\,d\,M_d}{a(1-\epsilon) } \,\frac{2^{d+4} \,\Gamma(\tfrac{d}{2}+1)^2}{\big(\frac{d}{2}+1\big)^{d/2}} \,+\,\sqrt{\frac{d+2}{2}}\right)\frac{1}{t}\leq R\,.
\end{equation*}
Since the summand $\sqrt{\frac{d+2}{2}}$ is dominated by the term on its left, we simply absorb it by replacing the factor $2^{d+4}$ with $2^{d+5}$. 
We now could choose $R\geq \frac{\widetilde{V}_d}{t}$, where 
\begin{equation*}
\widetilde{V}_d = \frac{A\,e^{\frac{1}{12}}\,d\,M_d}{a(1-\epsilon)} \,\frac{2^{d+5} \,\Gamma(\tfrac{d}{2}+1)^2}{\big(\frac{d}{2}+1\big)^{d/2}}\,.
\end{equation*}
For the further elaboration of the constants, we use Stirling's
formula for the Gamma function. 

\subsubsection*{Stirling's estimate for $\frac{\Gamma(\frac{d}{2}+1)^2}{(\frac{d}{2}+1)^{d/2}}$:\\}
The constant $V_d$ in the theorem results from majorizing  $\frac{\Gamma(\frac{d}{2}+1)^2}{(\frac{d}{2}+1)^{d/2}}$.
Evaluating Stirling with remainder, 
\begin{equation*}
\Gamma(x+1)\leq \sqrt{2\pi x}\Big(\frac{x}{e}\Big)^{x} e^{\frac{1}{12x}}\,,\quad x>0\,,
\end{equation*}
at $x=d/2$ leads to 
\begin{align*}
\frac{\Gamma(\tfrac{d}{2}+1)^2}{\big(\frac{d}{2}+1\big)^{d/2}} & \leq \pi d \, \frac{\big(\frac{d}{2}\big)^{\frac{d}{2}}\, \big(\frac{d}{2}\big)^{\frac{d}{2}}e^{-d} e^{\frac{1}{3d}}}{\big(\frac{d}{2}+1\big)^{d/2}}\\
& \leq \pi d \, 2^{-\frac{d}{2}} (\sqrt{d}/e)^d\, e^{\frac{1}{3d}}\,.
\end{align*}
We substitute this expression  into $\widetilde{V}_d$ and obtain
\begin{align*}
\widetilde{V}_d 
& \leq \frac{A/a}{(1-\epsilon)} \,\pi \,M_d\,d^2 \,2^{\frac{d}{2}+5}\, e^{\frac{1}{12}+\frac{1}{3d}}\,(\sqrt{d}/e)^{d}\,.
\end{align*}
Since $e^{\frac{1}{12}+\frac{1}{3d}}\leq 2$, the right-hand side is bounded by $V_d$ in \eqref{eq:Vd}.

\smallskip
The auxiliary assumptions $t \geq d+2$ and $t \geq 2d$ made through the above derivations can be discarded since the lower volume estimate is only nontrivial for $R < \pi$. In this case we have $\pi t \geq V_d$, and this condition is already much stronger than the auxiliary ones. This completes the proof.
\end{proof}

 \bibliographystyle{abbrv}
% \bibliography{general,new}

\end{document}